\documentclass[11pt]{article}
\usepackage{amsmath, amsfonts, amssymb, amsthm}
\usepackage{graphicx}
\usepackage{caption, subcaption}

\newtheorem{thm}{Theorem}
\newtheorem{prop}[thm]{Proposition}
\newtheorem{lem}[thm]{Lemma}

\newtheorem{rem}[thm]{Remark}

\renewcommand{\epsilon}{\varepsilon}
\renewcommand{\phi}{\varphi}

\newcommand{\BB}{\mathbb}

\newcommand{\separate}{\vskip5pt}

\newcommand{\tr}{\operatorname{Tr}}

\newcommand{\HC}{\BB H_{\BB C}}
\newcommand{\HR}{\BB H_{\BB R}}

\usepackage[OT2,T1]{fontenc}
\newcommand\textcyr[1]{{\fontencoding{OT2}\fontfamily{wncyr}\selectfont #1}}
\newcommand{\Zh}{\textit{\textcyr{Zh}}}

\begin{document}

\title{\bf The Second Order Pole over Split Quaternions}
\author{Matvei Libine}
\maketitle

\begin{abstract}
This is an addition to a series of papers \cite{FL1, FL2, FL3, FL4},
where we develop quaternionic analysis from the point of view of
representation theory of the conformal Lie group and its Lie algebra.
In this paper we develop split quaternionic analogues of certain results
from \cite{FL4}.
Thus we introduce a space of functions ${\cal D}^h \oplus {\cal D}^a$
with a natural action of the Lie algebra
$\mathfrak{gl}(2,\HC) \simeq \mathfrak{sl}(4,\BB C)$,
decompose ${\cal D}^h \oplus {\cal D}^a$ into irreducible components
and find the $\mathfrak{gl}(2,\HC)$-equivariant projectors onto each of these
irreducible components.
\end{abstract}

\section{Introduction}

This is an addition to a series of papers \cite{FL1, FL2, FL3, FL4},
where we develop quaternionic analysis from the point of view of
representation theory of the conformal Lie group and its Lie algebra.
Since the paper \cite{FL4} serves as a starting point for this article,
we highlight the most relevant results.
Let $\BB H$ denote the algebra of quaternions,
$\HC = \BB H \otimes_{\BB R} \BB C$ and
$$
\HC^{\times} = \{ Z \in \HC ;\: \text{$Z$ is invertible}\},
$$
then let
$$
\Zh = \BB C [z_{11},z_{12},z_{21},z_{22},N(Z)^{-1}]
$$
be the space of polynomial functions on $\HC^{\times}$.
We can think of $\Zh$ as a quaternionic analogue of Laurent polynomials
$\BB C[z,z^{-1}]$.
There is a natural action of the Lie algebra
$\mathfrak{gl}(2,\HC) \simeq \mathfrak{sl}(4,\BB C)$ on $\Zh$ denoted by
$\rho_1$ (see Lemma \ref{rho-algebra-action}).
According to Theorem 7 in \cite{FL4} (restated here as Theorem \ref{thm7}),
$(\rho_1,\Zh)$ decomposes into three irreducible components:
$$
\Zh = \Zh^- \oplus \Zh^0 \oplus \Zh^+.
$$
This decomposition is followed by a derivation of projectors of $\Zh$ onto
its irreducible components $\Zh^-$, $\Zh^0$ and $\Zh^+$.
These projectors are regarded as quaternionic analogues of
the Cauchy's formula for the second order pole
$$
f'(z_0) = \frac 1{2\pi i} \oint \frac {f(z)\,dz}{(z-z_0)^2}.
$$
The cases of $\Zh^-$ and $\Zh^+$ are fairly straightforward.
On the other hand, the case of $\Zh^0$ is more subtle and requires a certain
regularization of infinities which is well known in four-dimensional quantum
field theory as ``scalar vacuum polarization''.
This vacuum polarization is usually encoded by a Feynman diagram shown in
Figure \ref{vacuum} and plays a key role in renormalization theory
(see, for example, \cite{Sm}).

\begin{figure}
\begin{center}
\includegraphics[scale=1]{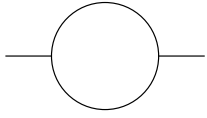}
\end{center}
\caption{Scalar vacuum polarization diagram}
\label{vacuum}
\end{figure}

Corollary 6 in \cite{FL4} (restated here as Proposition \ref{cor6})
asserts that
$$
\Zh = \BB C \text{-span of }
\bigl\{ t^l_{n\,\underline{m}}(Z) \cdot N(Z)^k;\: k \in \BB Z \bigr\},
$$
where $t^l_{n\,\underline{m}}(Z)$'s are the matrix coefficients of $SU(2)$.
In this article we replace the matrix coefficients of $SU(2)$ with the
matrix coefficients of the discrete series representations of $SU(1,1)$
-- denoted by $\tau^l_{n\,\underline{m}}(Z)$ -- and consider a space
\begin{equation}  \label{split-Zh}
{\cal D}^h \oplus {\cal D}^a = \BB C \text{-span of }
\bigl\{ \tau^l_{n\,\underline{m}}(Z) \cdot N(Z)^k;\: k \in \BB Z \bigr\}.
\end{equation}
(We denote by ${\cal D}^h$ and ${\cal D}^a$ the subspaces of (\ref{split-Zh})
spanned by $\tau^l_{n\,\underline{m}}(Z) \cdot N(Z)^k$, where
$\tau^l_{n\,\underline{m}}(Z)$'s are the matrix coefficients of holomorphic and
antiholomorphic discrete series respectively.)
We regard the space (\ref{split-Zh}) as a split quaternionic analogue of $\Zh$.
Our first result is Theorem \ref{decomposition-thm}, which shows that this
space is invariant under the $\rho_1$-action of $\mathfrak{gl}(2,\HC)$ and
decomposes into six irreducible components that we denote by
${\cal D}^h_<$, ${\cal D}^{--}$, ${\cal D}^h_>$, ${\cal D}^a_<$,
${\cal D}^{++}$ and ${\cal D}^a_>$.
Then in Proposition \ref{kernel-prop} we find the reproducing kernels for
each of these irreducible components.
Finally, in Theorem \ref{projector-thm} we obtain the
$\mathfrak{gl}(2,\HC)$-equivariant projectors onto the irreducible components
of the space (\ref{split-Zh}). This result can be regarded as a
split analogue of Corollary 14 and Theorem 15 in \cite{FL4}.
Perhaps surprisingly, unlike the case of $\Zh^0$, there are absolutely no
convergence issues for the reproducing kernels or the integral operators
realizing the projectors onto the irreducible components.

\section{Preliminaries}  \label{preliminaries-section}

Let $\HC$ denote the space of complexified quaternions:
$\HC = \BB H \otimes \BB C$, it can be identified with the algebra of
$2 \times 2$ complex matrices:
$$
\HC = \BB H \otimes \BB C \simeq \biggl\{
Z= \begin{pmatrix} z_{11} & z_{12} \\ z_{21} & z_{22} \end{pmatrix}
; \: z_{ij} \in \BB C \biggr\}
= \biggl\{ Z= \begin{pmatrix} z^0-iz^3 & -iz^1-z^2 \\ -iz^1+z^2 & z^0+iz^3
\end{pmatrix} ; \: z^k \in \BB C \biggr\}.
$$
For $Z \in \HC$, we write
$$
N(Z) = \det \begin{pmatrix} z_{11} & z_{12} \\ z_{21} & z_{22} \end{pmatrix}
= z_{11}z_{22}-z_{12}z_{21} = (z^0)^2 + (z^1)^2 + (z^2)^2 + (z^3)^2
$$
and think of it as the norm of $Z$.
Let
$$
\HC^{\times} = \{ Z \in \HC ;\: N(Z) \ne 0 \}
$$
be the group of invertible complexified  quaternions,
$\HC^{\times} \simeq GL(2,\BB C)$.

Let $\widetilde{\Zh}$ denote the space of $\BB C$-valued functions on $\HC$
(possibly with singularities) which are holomorphic with respect to the
complex variables $z_{11},z_{12},z_{21},z_{22}$.
We recall the action of $GL(2,\HC)$ on $\widetilde{\Zh}$ given by
equation (49) in \cite{FL1}:
\begin{multline*}
\rho_1(h): \: f(Z) \quad \mapsto \quad \bigl( \rho_1(h)f \bigr)(Z) =
\frac {f \bigl( (aZ+b)(cZ+d)^{-1} \bigr)}{N(cZ+d) \cdot N(a'-Zc')},  \\
h = \begin{pmatrix} a' & b' \\ c' & d' \end{pmatrix},\:
h^{-1} = \begin{pmatrix} a & b \\ c & d \end{pmatrix} \in GL(2, \HC).
\end{multline*}
Differentiating this action, we obtain an action of the Lie algebra
$\mathfrak{gl}(2,\HC) \simeq \mathfrak{gl}(4,\BB C)$,
which we still denote by $\rho_1$.

\begin{lem}[Lemma 68 from \cite{FL1}]  \label{rho-algebra-action}
Let $\partial = \bigl(\begin{smallmatrix} \partial_{11} & \partial_{21} \\
\partial_{12} & \partial_{22} \end{smallmatrix}\bigr)$, where
$\partial_{ij} = \frac{\partial}{\partial z_{ij}}$.
The Lie algebra action $\rho_1$ of $\mathfrak{gl}(2,\HC)$ on $\widetilde{\Zh}$
is given by
\begin{align*}
\rho_1 \begin{pmatrix} A & 0 \\ 0 & 0 \end{pmatrix} &:
f \mapsto \tr \bigl( A \cdot (-Z \cdot \partial f - f) \bigr)  \\
\rho_1 \begin{pmatrix} 0 & B \\ 0 & 0 \end{pmatrix} &:
f \mapsto \tr \bigl( B \cdot (-\partial f ) \bigr)  \\
\rho_1 \begin{pmatrix} 0 & 0 \\ C & 0 \end{pmatrix} &:
f \mapsto \tr \Bigl( C \cdot \bigl(
Z \cdot (\partial f) \cdot Z +2Zf \bigr) \Bigr)
= \tr \Bigl( C \cdot \bigl(Z \cdot \partial (Zf) \bigr) \Bigr)  \\
\rho_1 \begin{pmatrix} 0 & 0 \\ 0 & D \end{pmatrix} &:
f \mapsto \tr \Bigl( D \cdot \bigl( (\partial f) \cdot Z + f \bigr) \Bigr)
= \tr \Bigl( D \cdot \bigl( \partial (Zf) - f \bigr) \Bigr).
\end{align*}
\end{lem}

This lemma implies that $\mathfrak{gl}(2, \HC)$ preserves the spaces
\begin{align*}
\Zh^+ &= \{\text{polynomial functions on $\HC$}\}
= \BB C[z_{11},z_{12},z_{21},z_{22}]
\qquad \qquad \text{and} \\
\Zh &= \{\text{polynomial functions on $\HC^{\times}$}\}
= \BB C[z_{11},z_{12},z_{21},z_{22},N(Z)^{-1}].
\end{align*}

Recall that $t^l_{n\,\underline{m}}(Z)$'s denote the matrix coefficients of $SU(2)$
described by equation (27) of \cite{FL1} (cf. \cite{V}).
These are polynomial functions on $\HC$.
The parameters range as follows:
$$
l=0,\frac 12, 1, \frac 32,\dots, \qquad
m,n= -l, -l+1, -l+2, \dots, l.
$$
We recall some results from \cite{FL4}.

\begin{prop}[Corollary 6 in \cite{FL4}]  \label{cor6}
The functions 
$$
t^l_{n\,\underline{m}}(Z) \cdot N(Z)^k, \qquad
l=0, \frac12, 1, \frac32, \dots, \quad m,n=-l,-l+1,\dots,l, \quad k \in\BB Z,
$$
form a vector space basis of $\Zh = \BB C[z_{11},z_{12},z_{21},z_{22},N(Z)^{-1}]$.
\end{prop}

\begin{thm}[Theorem 7 in \cite{FL4}]  \label{thm7}
The representation $(\rho_1,\Zh)$ of $\mathfrak{gl}(2,\HC)$ has the
following decomposition into irreducible components:
$$
(\rho_1,\Zh) = (\rho_1,\Zh^-) \oplus (\rho_1,\Zh^0) \oplus (\rho_1,\Zh^+),
$$
where
\begin{align*}
\Zh^+ &= \BB C \text{-span of }
\bigl\{ t^l_{n\,\underline{m}}(Z) \cdot N(Z)^k;\: k \ge 0 \bigr\}, \\
\Zh^- &= \BB C \text{-span of }
\bigl\{ t^l_{n\,\underline{m}}(Z) \cdot N(Z)^k;\: k \le -(2l+2) \bigr\}, \\
\Zh^0 &= \BB C \text{-span of }
\bigl\{ t^l_{n\,\underline{m}}(Z) \cdot N(Z)^k;\: -(2l+1) \le k \le -1 \bigr\}
\end{align*}
(see Figure \ref{decomposition-fig}).
\end{thm}

\begin{figure}
\begin{center}
\setlength{\unitlength}{.7mm}
\begin{picture}(110,70)
\multiput(10,10)(10,0){10}{\circle*{1}}
\multiput(10,20)(10,0){10}{\circle*{1}}
\multiput(10,30)(10,0){10}{\circle*{1}}
\multiput(10,40)(10,0){10}{\circle*{1}}
\multiput(10,50)(10,0){10}{\circle*{1}}
\multiput(10,60)(10,0){10}{\circle*{1}}
\thicklines
\put(60,0){\vector(0,1){70}}
\put(0,10){\vector(1,0){110}}
\thinlines
\put(58,10){\line(0,1){55}}
\put(60,8){\line(1,0){45}}
\put(52,10){\line(0,1){55}}
\put(5,8){\line(1,0){35}}
\put(41.4,11.4){\line(-1,1){36.4}}
\put(48.6,8.6){\line(-1,1){43.6}}
\qbezier(58,10)(58,8)(60,8)
\qbezier(40,8)(43.8,8)(41.4,11.4)
\qbezier(52,10)(52,5.2)(48.6,8.6)
\put(62,67){$2l$}
\put(107,12){$k$}
\put(3,23){\small $\Zh^-$}
\put(23,63){\small $\Zh^0$}
\put(102,43){\small $\Zh^+$}
\end{picture}
\end{center}
\caption{Decomposition of $(\rho_1,\Zh)$ into irreducible components}
\label{decomposition-fig}
\end{figure}
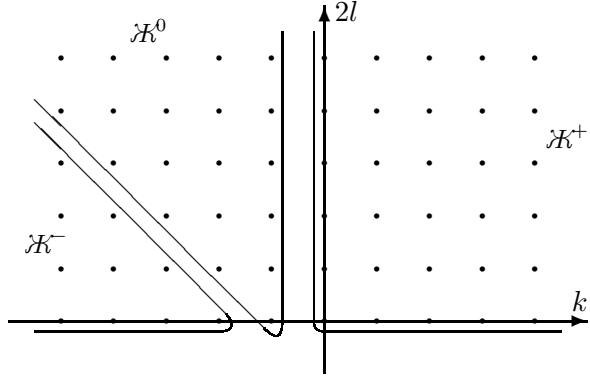

\section{Split Quaternions}
\label{split-section}

In this section we essentially replace the classical quaternions $\BB H$ with
split quaternions $\HR$ by replacing the matrix coefficients
$t^l_{m \, \underline{n}}(Z)$ of $SU(2)$ with the matrix coefficients of $SU(1,1)$,
which we denote by $\tau^l_{m \, \underline{n}}(Z)$ in order to avoid conflict
of notations.

We are primarily interested in a real form of $\HC$ called split quaternions:
$$
\HR = \biggl\{
Z= \begin{pmatrix} z_{11} & z_{12} \\ z_{21} & z_{22} \end{pmatrix} \in \HC
; \: z_{22} = \overline{z_{11}}, \: z_{21} = \overline{z_{12}} \biggr\}.
$$
This algebra is studied in detail in \cite{FL2}.
The ``unit sphere'' in $\HR$ is a hyperboloid
$$
SU(1,1) = \{X \in \HR ;\: N(X)=1 \},
$$
which is a Lie group isomorphic to $SL(2,\BB R)$.
Another important group sitting inside $\HC$ is
$$
U(1,1) = \{ Z=e^{i\theta} \cdot X \in \HC;\:
X \in SU(1,1) \subset \HR,\: \theta \in \BB R \}.
$$
We orient $U(1,1)$ so that $\{\tilde e_0, \tilde e_1, \tilde e_2, e_3\}$
is a positive basis of the tangent space at $1 \in U(1,1)$
(same orientation as in \cite{FL2}), where
$$
\tilde e_0 = \begin{pmatrix} -i & 0 \\ 0 & -i \end{pmatrix}, \quad
\tilde e_1 = \begin{pmatrix} 0 & 1 \\ 1 & 0 \end{pmatrix}, \quad
\tilde e_2 = \begin{pmatrix} 0 & i \\ -i & 0 \end{pmatrix}, \quad
e_3 = \begin{pmatrix} -i & 0 \\ 0 & i \end{pmatrix}.
$$
We regard groups $U(1,1) \times U(1,1)$ and $SU(1,1) \times SU(1,1)$
as subgroups of $GL(2,\HC)$:
\begin{equation}  \label{U(1,1)xU(1,1)}
U(1,1) \times U(1,1) = \biggl\{
\begin{pmatrix} a & 0 \\ 0 & d \end{pmatrix} \in GL(2, \HC);\:
a,d \in U(1,1) \subset \HC \biggr\}.
\end{equation}
\begin{equation}  \label{SU(1,1)xSU(1,1)}
SU(1,1) \times SU(1,1) = \biggl\{
\begin{pmatrix} a & 0 \\ 0 & d \end{pmatrix} \in GL(2, \HC);\:
a,d \in \HR, \: N(a)=N(d)=1 \biggr\}.
\end{equation}

The matrix coefficients of the discrete series representations of $SU(1,1)$
are rational functions on $\HC$ described by equation (13) of \cite{FL2}
(cf. \cite{V}):
\begin{equation}  \label{int_t}
\tau^l_{n\,\underline{m}}(Z) = \frac 1{2\pi i}
\oint (sz_{11}+z_{21})^{l-m} (sz_{12}+z_{22})^{l+m} s^{-l+n} \frac {ds}s,
\qquad Z = \begin{pmatrix} z_{11} & z_{12} \\ z_{21} & z_{22} \end{pmatrix},
\end{equation}
where the integral is taken over the unit circle $\{ s\in \BB C ;\: |s|=1 \}$
traversed once in the counterclockwise direction.
The parameters $l$, $m$ and $n$ range over
$$
l = -1, -\frac 32, -2, -\frac 52, \dots,
\qquad m,n \in \BB Z + l, \qquad m, n \le l \:\: \text{or} \:\: m, n \ge -l.
$$
When $m, n \ge -l$, we get the matrix coefficients of the holomorphic
discrete series representations.
And when $m, n \le l$, we get the matrix coefficients of the antiholomorphic
discrete series representations.
While the integral expressions for  $t^l_{n\,\underline{m}}(Z)$'s and
$\tau^l_{n\,\underline{m}}(Z)$'s are exactly the same, the indices $l$, $m$ and $n$
have different ranges, and we insist on using different notations, since
the former are matrix coefficients of $SU(2)$ and the latter are matrix
coefficients of $SU(1,1)$.

We define split analogues of the space $\Zh$:
\begin{align*}
{\cal D}^h &= \BB C \text{-span of } \bigl\{
\tau^l_{n\,\underline{m}}(Z) \cdot N(Z)^k;\: m, n \ge -l, \: k \in \BB Z \bigr\}, \\
{\cal D}^a &= \BB C \text{-span of } \bigl\{
\tau^l_{n\,\underline{m}}(Z) \cdot N(Z)^k;\: m, n \le l, \: k \in \BB Z \bigr\}
\end{align*}
($h$ stands for ``holomorphic'' and $a$ stands for ``antiholomorphic'').

As in Subsection 3.6 in \cite{FL2}, we define a $\mathfrak{gl}(2,\HC)$-invariant
symmetric bilinear pairing for functions on ${\cal D}^h \oplus {\cal D}^a$
\begin{equation}  \label{pairing}
\langle f_1,f_2 \rangle_1 = \frac i{2\pi^3} \int_{U(1,1)} f_1(Z) \cdot f_2(Z) \,dV,
\end{equation}
where $dV$ is a holomorphic 4-form on $\HC$
$$
dV = dz^0 \wedge dz^1 \wedge dz^2 \wedge dz^3
= \frac14 dz_{11} \wedge dz_{12} \wedge dz_{21} \wedge dz_{22}.
$$

\begin{prop}[Proposition 56 in \cite{FL2}]
We have the following orthogonality relations:
\begin{equation}  \label{orthogonality}
\bigl\langle \tau^{l'}_{n'\,\underline{m'}}(Z) \cdot N(Z)^{k'},
\tau^l_{m\underline{n}}(Z^{-1}) \cdot N(Z)^{-k-2} \bigr\rangle_1
= - \frac1{2l+1} \delta_{kk'}\delta_{ll'} \delta_{mm'} \delta_{nn'},
\end{equation}
where the indices $k,l,m,n$ are
$l = -1, -\frac 32, -2, \dots$, $m,n \in \BB Z +l$, $m,n \ge -l$ or $m,n \le l$,
$k \in \BB Z$ and similarly for $k',l',m',n'$.
In particular, the functions $\tau^l_{n\,\underline{m}}(Z)\cdot N(Z)^k$'s
are linearly independent.
\end{prop}

Since $\tau^l_{m\underline{n}}(Z^{-1}) = \tau^l_{m\underline{n}}(Z^+) \cdot N(Z)^{-2l}$
is proportional to $\tau^l_{-n\underline{-m}}(Z) \cdot N(Z)^{-2l}$,
the functions $\tau^l_{m\underline{n}}(Z^{-1}) \cdot N(Z)^{-k-2}$ indeed lie in
${\cal D}^h \oplus {\cal D}^a$. Hence the spaces ${\cal D}^h$ and ${\cal D}^a$
are dual to each other under the pairing (\ref{pairing}).

\section{Decomposition of ${\cal D}^h$ and ${\cal D}^a$ into
Irreducible Components}

In this section we discuss a split analogue of Theorem \ref{thm7}:

\begin{thm}  \label{decomposition-thm}
The spaces ${\cal D}^h$ and ${\cal D}^a$ are invariant under the
$\rho_1$-action of $\mathfrak{gl}(2,\HC)$. We have the following
decompositions into irreducible components:
$$
(\rho_1,{\cal D}^h) =
(\rho_1,{\cal D}^h_<) \oplus (\rho_1,{\cal D}^{--}) \oplus (\rho_1,{\cal D}^h_>)
\quad \text{and} \quad
(\rho_1,{\cal D}^a) =
(\rho_1,{\cal D}^a_<) \oplus (\rho_1,{\cal D}^{++}) \oplus (\rho_1,{\cal D}^a_>),
$$
where
\begin{align*}
{\cal D}^h_< &= \BB C \text{-span of }
\bigl\{ \tau^l_{n\,\underline{m}}(Z) \cdot N(Z)^k;\:
m, n \ge -l,\: k \le -1 \bigr\}, \\
{\cal D}^{--} &= \BB C \text{-span of }
\bigl\{ \tau^l_{n\,\underline{m}}(Z) \cdot N(Z)^k;\:
m, n \ge -l,\: 0 \le k \le -(2l+2) \bigr\}, \\
{\cal D}^h_> &= \BB C \text{-span of }
\bigl\{ \tau^l_{n\,\underline{m}}(Z) \cdot N(Z)^k;\:
m, n \ge -l,\: k \ge -(2l+1) \bigr\}, \\
{\cal D}^a_< &= \BB C \text{-span of }
\bigl\{ \tau^l_{n\,\underline{m}}(Z) \cdot N(Z)^k;\:
m, n \le l,\: k \le -1 \bigr\}, \\
{\cal D}^{++} &= \BB C \text{-span of }
\bigl\{ \tau^l_{n\,\underline{m}}(Z) \cdot N(Z)^k;\:
m, n \le l,\: 0 \le k \le -(2l+2) \bigr\}, \\
{\cal D}^a_> &= \BB C \text{-span of }
\bigl\{ \tau^l_{n\,\underline{m}}(Z) \cdot N(Z)^k;\:
m, n \le l,\: k \ge -(2l+1) \bigr\}
\end{align*}
(see Figure \ref{decomposition-fig2}).
\end{thm}

\begin{figure}
\begin{center}
\setlength{\unitlength}{.7mm}
\begin{picture}(110,70)
\multiput(10,10)(10,0){10}{\circle*{1}}
\multiput(10,20)(10,0){10}{\circle*{1}}
\multiput(10,30)(10,0){10}{\circle*{1}}
\multiput(10,40)(10,0){10}{\circle*{1}}
\thicklines
\put(50,0){\vector(0,1){70}}
\put(0,60){\vector(1,0){110}}
\thinlines
\put(5,42){\line(1,0){35}}
\put(42,5){\line(0,1){35}}
\qbezier(40,42)(42,42)(42,40)

\put(48,5){\line(0,1){35}}
\put(51.4,41,4){\line(1,-1){36.4}}
\qbezier(48,40)(48,44.8)(51.4,41.4)

\put(60,42){\line(1,0){45}}
\put(58.6,38.6){\line(1,-1){33.6}}
\qbezier(60,42)(55.2,42)(58.6,38.6)

\put(52,67){$2l$}
\put(106,62){$k$}
\put(7,45){\small ${\cal D}^h_<$ and ${\cal D}^a_<$}
\put(70,45){\small ${\cal D}^h_>$ and ${\cal D}^a_>$}
\put(53,1){\small ${\cal D}^{--}$ and ${\cal D}^{++}$}
\end{picture}
\end{center}
\caption{Decompositions of $(\rho_1,{\cal D}^h)$ and $(\rho_1,{\cal D}^a)$
into irreducible components}
\label{decomposition-fig2}
\end{figure}
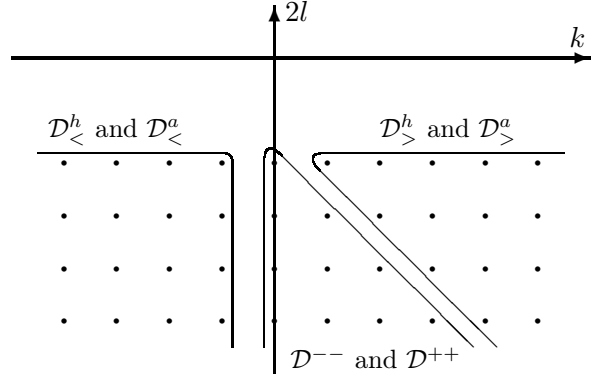





\begin{proof}
Since the proof is essentially the same as that of Theorem 7 in \cite{FL4},
we only give an outline.
Note that the basis elements of ${\cal D}^h$ and ${\cal D}^a$
consist of functions of the kind
$$
f_l(Z) \cdot N(Z)^k, \qquad \square f_l(Z)=0, \quad
l= -1, -\frac32, -2, -\frac52, \dots, \quad k \in\BB Z,
$$
where the functions $f_l(Z)$ range over a basis of harmonic functions which
are homogeneous of degree $2l$.
Recall that we consider $U(1,1) \times U(1,1)$ and $SU(1,1) \times SU(1,1)$
as subgroups of $GL(2,\HC)$ via (\ref{U(1,1)xU(1,1)}) and
(\ref{SU(1,1)xSU(1,1)}).
For $k$ and $l$ fixed, these functions span an irreducible representation
of the Lie algebra $\mathfrak{u}(1,1) \times \mathfrak{u}(1,1)$, which remains
irreducible when restricted to $\mathfrak{su}(1,1) \times \mathfrak{su}(1,1)$,
belongs to the holomorphic discrete series when $m,n \ge -l$
and the antiholomorphic discrete series when $m,n \le l$.

\begin{figure}
\begin{center}
\begin{subfigure}[b]{0.2\textwidth}
\centering
\includegraphics[scale=0.17]{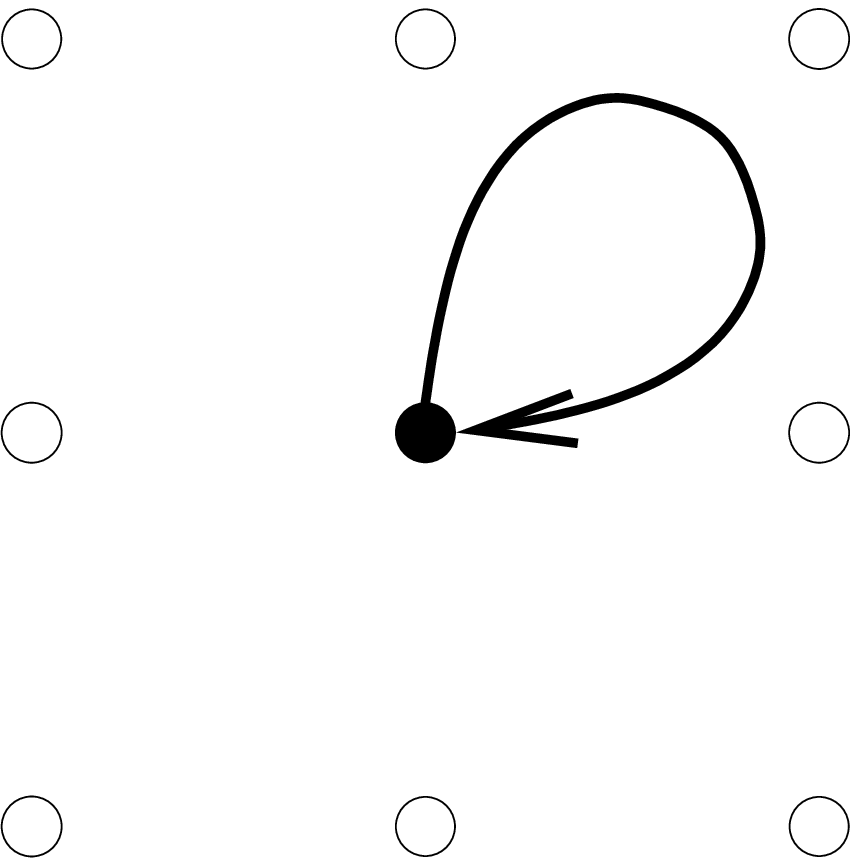}
\caption{\footnotesize Action of
$\rho_1\bigl(\begin{smallmatrix} A & 0 \\ 0 & 0 \end{smallmatrix}\bigr)$}
\end{subfigure}
\quad
\begin{subfigure}[b]{0.2\textwidth}
\centering
\includegraphics[scale=0.17]{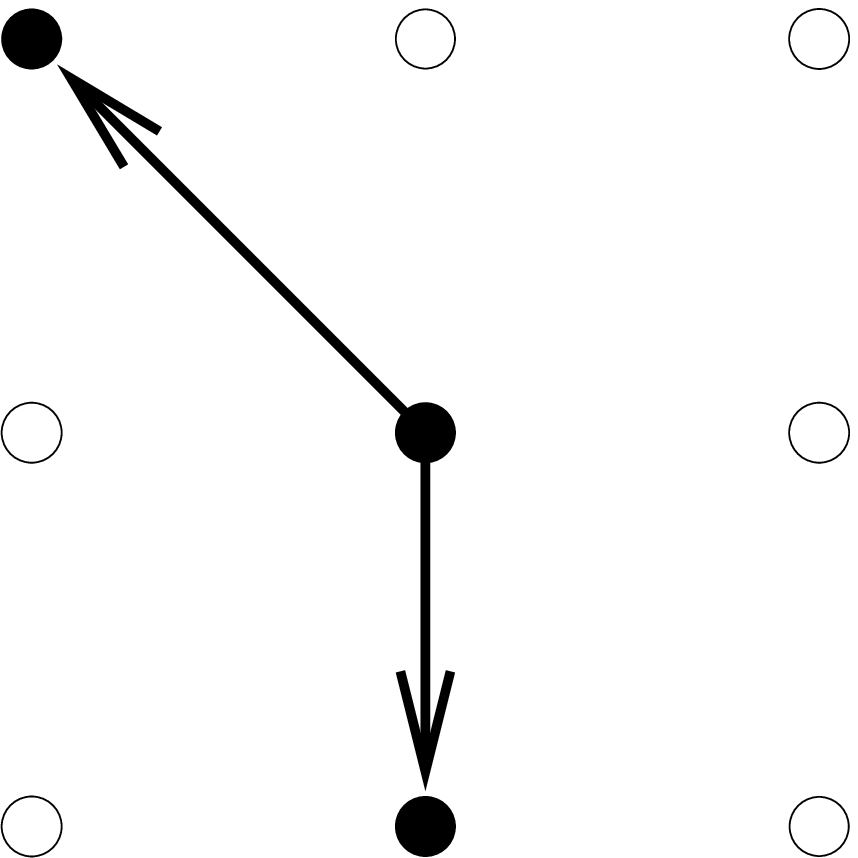}
\caption{\footnotesize Action of
$\rho_1\bigl(\begin{smallmatrix} 0 & B \\ 0 & 0 \end{smallmatrix}\bigr)$}
\end{subfigure}
\quad
\begin{subfigure}[b]{0.2\textwidth}
\centering
\includegraphics[scale=0.17]{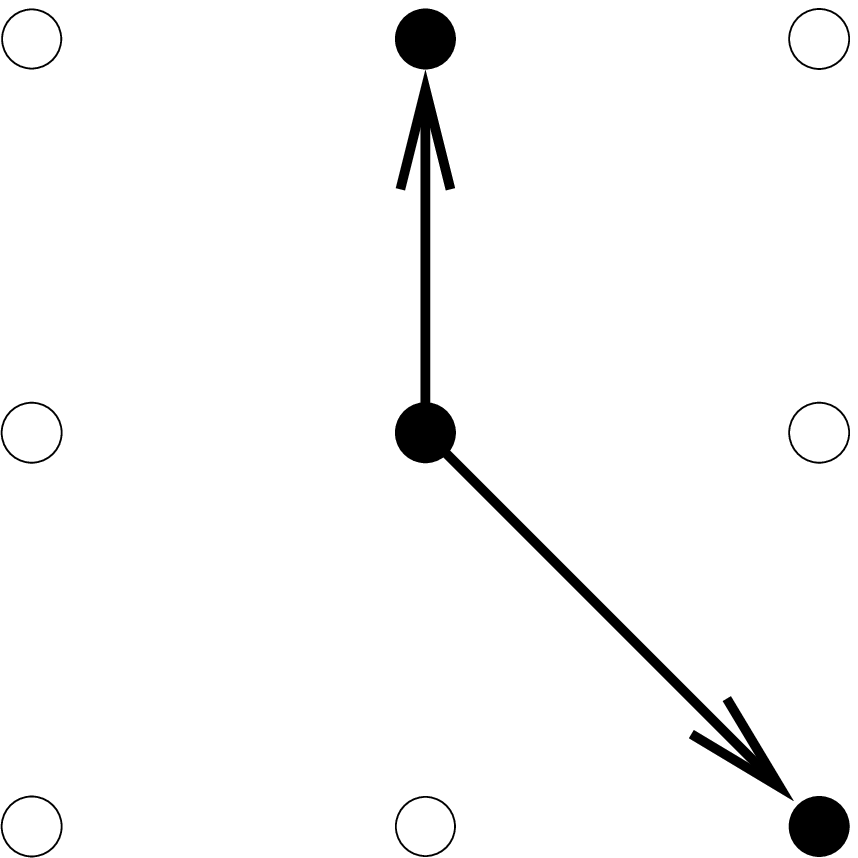}
\caption{\footnotesize Action of
$\rho_1\bigl(\begin{smallmatrix} 0 & 0 \\ C & 0 \end{smallmatrix}\bigr)$}
\end{subfigure}
\quad
\begin{subfigure}[b]{0.2\textwidth}
\centering
\includegraphics[scale=0.17]{AD.eps}
\caption{\footnotesize Action of
$\rho_1\bigl(\begin{smallmatrix} 0 & 0 \\ 0 & D \end{smallmatrix}\bigr)$}
\end{subfigure}
\end{center}
\caption{}
\label{actions}
\end{figure}

By Lemma \ref{rho-algebra-action}, matrices of the kind
$\bigl( \begin{smallmatrix} 0 & B \\ 0 & 0 \end{smallmatrix} \bigr)
\in \mathfrak{gl}(2,\HC)$ with $B \in \HC$ act by
$$
f \mapsto \tr \bigl( B \cdot (-\partial f ) \bigr),
$$
and, by equation (27) in \cite{FL4},
$$
\partial \bigl( f_l(Z) \cdot N(Z)^k \bigr)
= \frac{2l+k+1}{2l+1} \partial f_l \cdot N(Z)^k
+ \frac{k}{2l+1} \bigl( Z^+ \cdot (\partial^+ f_l) \cdot Z^+ + Z^+f_l \bigr)
\cdot N(Z)^{k-1}
$$
with $\partial f_l$ and $Z^+ \cdot (\partial^+ f_l) \cdot Z^+ + Z^+f_l$
being harmonic and having degrees $2l-1$ and $2l+1$ respectively.

Using Lemma \ref{rho-algebra-action} again, matrices of the kind
$\bigl( \begin{smallmatrix} 0 & 0 \\ C & 0 \end{smallmatrix} \bigr)
\in \mathfrak{gl}(2,\HC)$ with
$C \in \HC$ act by
$$
f \mapsto \tr \bigl( C \cdot \bigl(
Z \cdot (\partial f) \cdot Z +2Zf \bigr) \bigr)
= \tr \bigl( C \cdot \bigl(Z \cdot \partial (Zf) \bigr) \bigr),
$$
and, by equation (28) in \cite{FL4},

\begin{multline*}
Z \cdot \partial \bigl( f_l \cdot N(Z)^k \bigr) \cdot Z + 2 Zf_l \cdot N(Z)^k \\
= \frac{2l+k+2}{2l+1} \bigl( Z \cdot (\partial f_l) \cdot Z + Zf_l \bigr)
\cdot N(Z)^k + \frac{k+1}{2l+1} \partial^+ f_l \cdot N(Z)^{k+1}
\end{multline*}
with $Z \cdot (\partial f_l) \cdot Z + Zf_l$ and $\partial^+ f_l$
being harmonic and having degrees $2l+1$ and $2l-1$ respectively.

Since
$$
\partial \bigl( f(Z^{-1}) \bigr)
= - Z^{-1} \cdot (\partial f) \bigr|_{Z^{-1}} \cdot Z^{-1},
$$
it follows that
$$
Z \cdot (\partial f_l) \cdot Z + Zf_l =
- \operatorname{Inv} \circ \partial \circ \operatorname{Inv},
$$
where
$$
\operatorname{Inv}: \quad f \mapsto \frac1{N(Z)} \cdot f(Z^{-1}).
$$
Hence $Z \cdot (\partial f_l) \cdot Z + Zf_l$
and its conjugate $Z^+ \cdot (\partial^+ f_l) \cdot Z^+ + Z^+f_l$
become zero if and only if $l=-1$.

The actions of
$\bigl( \begin{smallmatrix} A & 0 \\ 0 & 0 \end{smallmatrix} \bigr)$,
$\bigl( \begin{smallmatrix} 0 & B \\ 0 & 0 \end{smallmatrix} \bigr)$,
$\bigl( \begin{smallmatrix} 0 & 0 \\ C & 0 \end{smallmatrix} \bigr)$ and
$\bigl( \begin{smallmatrix} 0 & 0 \\ 0 & D \end{smallmatrix} \bigr)$
are illustrated in Figure \ref{actions}. In the diagram describing
$\rho_1\bigl( \begin{smallmatrix} 0 & B \\ 0 & 0 \end{smallmatrix} \bigr)$
the vertical arrow disappears when $2l+k+1=0$
and the diagonal arrow disappears when $k=0$ or $l=-1$.
Similarly, in the diagram describing
$\rho_1\bigl( \begin{smallmatrix} 0 & 0 \\ C & 0 \end{smallmatrix} \bigr)$
the vertical arrow disappears when $l=-1$ or $2l+k+2=0$ and the
diagonal arrow disappears when $k=-1$.
This proves that the spaces ${\cal D}^h$, ${\cal D}^h_<$, ${\cal D}^{--}$,
${\cal D}^h_>$, ${\cal D}^a$, ${\cal D}^a_<$, ${\cal D}^{++}$ and ${\cal D}^a_>$
are $\mathfrak{gl}(2,\HC)$-invariant.
This also proves that ${\cal D}^h_<$, ${\cal D}^{--}$ and ${\cal D}^h_>$
are irreducible $\mathfrak{gl}(2,\HC)$-invariant subspaces of ${\cal D}^h$
and that ${\cal D}^a_<$, ${\cal D}^{++}$ and ${\cal D}^a_>$ are irreducible
$\mathfrak{gl}(2,\HC)$-invariant subspaces of ${\cal D}^a$.
\end{proof}

\begin{rem}
Theorem 53 in \cite{FL2} gives explicit isomorphisms of
$\mathfrak{gl}(2,\HC)$-modules
$$
(\rho_1,{\cal D}^{--}) \simeq (\rho_1, \Zh^-) \qquad \text{and} \qquad
(\rho_1,{\cal D}^{++}) \simeq (\rho_1, \Zh^+).
$$

One immediately notices the similarity of Figures \ref{decomposition-fig}
and \ref{decomposition-fig2} illustrating the decompositions of Theorems
\ref{thm7} and \ref{decomposition-thm} respectively.
The figures look similar because the matrix coefficients
$t^l_{n\,\underline{m}}(Z)$'s of $SU(2)$ and $\tau^l_{n\,\underline{m}}(Z)$'s of
$SU(1,1)$ span irreducible representations of the diagonal subalgebras for
$l$ fixed, are homogeneous of degree $2l$, except $l \ge 0$ in the case of
$SU(2)$-coefficients and $l \le -1$ in the case of $SU(1,1)$-coefficients.
The ``border conditions'' separating irreducible components $k \ge 0$ or
$k \le -1$ and $k \ge -(2l+1)$ or $k \le -(2l+2)$ are formally the same
in the cases of $\Zh$, ${\cal D}^h$ and ${\cal D}^a$.
\end{rem}

\section{Reproducing Kernels}

In this section we obtain the reproducing kernels for the irreducible
components of ${\cal D}^h$ and ${\cal D}^a$.
Recall that $\Gamma^-$ and $\Gamma^+$ are certain open Ol'shanskii semigroups
lying inside $GL(2,\BB C) \subset \HC$ which were described in
Subsection 3.3 in \cite{FL2}. For our purposes they can be defined as
$$
\Gamma^- =
\biggl\{ g \cdot \begin{pmatrix} \lambda_1 & 0 \\ 0 & \lambda_2 \end{pmatrix}
\cdot g^{-1} \in GL(2,\BB C);\: g \in SU(1,1),\:
\lambda_1, \lambda_2 \in \BB C,\: |\lambda_1| > 1 > |\lambda_2| > 0 \biggr\},
$$
$$
\Gamma^+ =
\biggl\{ g \cdot \begin{pmatrix} \lambda_1 & 0 \\ 0 & \lambda_2 \end{pmatrix}
\cdot g^{-1} \in GL(2,\BB C);\: g \in SU(1,1),\:
\lambda_1, \lambda_2 \in \BB C,\: |\lambda_2| > 1 > |\lambda_1| > 0 \biggr\}.
$$

We already have the reproducing kernels for ${\cal D}^{--}$ and ${\cal D}^{++}$:

\begin{prop}[Proposition 54 in \cite{FL2}]
We have the following matrix coefficient expansions for the reproducing kernels
of ${\cal D}^{--}$ and ${\cal D}^{++}$ respectively:
$$
\frac1{N(Z-W)^2} =
\sum_{\text{\tiny $\begin{matrix} k,l,m,n \\ m,n \ge -l \ge 1 \\ 0 \le k \le -2l-2 \end{matrix}$}}
-(2l+1) \tau^l_{n\underline{m}}(W) \cdot N(W)^k \cdot
\tau^l_{m\underline{n}}(Z^{-1}) \cdot N(Z)^{-k-2}
$$
which converges pointwise absolutely whenever $WZ^{-1} \in \Gamma^-$ and
$$
\frac1{N(Z-W)^2} =
\sum_{\text{\tiny $\begin{matrix} k,l,m,n \\ m,n \le l \le -1 \\ 0 \le k \le -2l-2 \end{matrix}$}}
-(2l+1) \tau^l_{n\underline{m}}(W) \cdot N(W)^k \cdot
\tau^l_{m\underline{n}}(Z^{-1}) \cdot N(Z)^{-k-2}
$$
which converges pointwise absolutely whenever $WZ^{-1} \in \Gamma^+$.
The sums are taken first over all $m$, $n$, then over $k=0,1,2,\dots, -(2l+2)$,
and then over $l=-1,-\frac32, -2, -\frac52,\dots$.
\end{prop}

It remains to find the reproducing kernels for the other four irreducible
components.

\begin{prop}  \label{kernel-prop}
Suppose that $WZ^{-1}$ is $SU(1,1)$-conjugate to a diagonal matrix
$\bigl(\begin{smallmatrix} \lambda_1 & 0 \\
0 & \lambda_2 \end{smallmatrix}\bigr)$.
We have the following matrix coefficient expansions for the reproducing kernels
of ${\cal D}^h_<$, ${\cal D}^h_>$, ${\cal D}^a_<$ and ${\cal D}^a_>$
respectively:
\begin{enumerate}
\item  \label{1}
When $WZ^{-1} \in \Gamma^-$ and $|N(WZ^{-1})|>1$ we have
a pointwise absolutely convergent series
$$
\sum_{\text{\tiny $\begin{matrix} k,l,m,n \\ m,n \ge -l \ge 1 \\ k \le -1 \end{matrix}$}}
-(2l+1) \tau^l_{n\underline{m}}(W) \cdot N(W)^k \cdot
\tau^l_{m\underline{n}}(Z^{-1}) \cdot N(Z)^{-k-2}
= \frac{- N(Z)^{-2} \cdot \lambda_1}
{(\lambda_1-\lambda_2)(1-\lambda_1 \lambda_2)(1-\lambda_1)^2};
$$
\item
When $WZ^{-1} \in \Gamma^-$ and $|N(WZ^{-1})|<1$ we have
a pointwise absolutely convergent series
$$
\sum_{\text{\tiny $\begin{matrix} k,l,m,n \\ m,n \ge -l \ge 1 \\ k \ge -(2l+1) \end{matrix}$}}
-(2l+1) \tau^l_{n\underline{m}}(W) \cdot N(W)^k \cdot
\tau^l_{m\underline{n}}(Z^{-1}) \cdot N(Z)^{-k-2}
= \frac{- N(Z)^{-2} \cdot \lambda_2}
{(\lambda_1-\lambda_2)(1-\lambda_1 \lambda_2)(1-\lambda_2)^2};
$$
\item
When $WZ^{-1} \in \Gamma^+$ and $|N(WZ^{-1})|>1$ we have
a pointwise absolutely convergent series
$$
\sum_{\text{\tiny $\begin{matrix} k,l,m,n \\ m,n \le l \le -1 \\ k \le -1 \end{matrix}$}}
-(2l+1) \tau^l_{n\underline{m}}(W) \cdot N(W)^k \cdot
\tau^l_{m\underline{n}}(Z^{-1}) \cdot N(Z)^{-k-2}
= \frac{N(Z)^{-2} \cdot \lambda_2}
{(\lambda_1-\lambda_2)(1-\lambda_1 \lambda_2)(1-\lambda_2)^2};
$$
\item
When $WZ^{-1} \in \Gamma^+$ and $|N(WZ^{-1})|<1$ we have
a pointwise absolutely convergent series
$$
\sum_{\text{\tiny $\begin{matrix} k,l,m,n \\ m,n \le l \le -1 \\ k \ge -(2l+1) \end{matrix}$}}
-(2l+1) \tau^l_{n\underline{m}}(W) \cdot N(W)^k \cdot
\tau^l_{m\underline{n}}(Z^{-1}) \cdot N(Z)^{-k-2}
= \frac{N(Z)^{-2} \cdot \lambda_1}
{(\lambda_1-\lambda_2)(1-\lambda_1 \lambda_2)(1-\lambda_1)^2}.
$$
\end{enumerate}
The sums are taken first over all $m$, $n$, then over $k$,
and then over $l=-1,-\frac32, -2, -\frac52,\dots$.
\end{prop}

\begin{proof}
We will prove part {\em \ref{1}} only, the other parts can be proved
in the same way. The proof is similar to that of Proposition 54 in \cite{FL2}.
Since $WZ^{-1} \in \Gamma^-$, $WZ^{-1}$ is $SU(1,1)$-conjugate to a diagonal
matrix $\bigl(\begin{smallmatrix} \lambda_1 & 0 \\ 0 & \lambda_2
\end{smallmatrix}\bigr)$ for some $\lambda_1, \lambda_2 \in \BB C$ with
$|\lambda_1| > 1 > |\lambda_2| > 0$.
Using equations (33) and (34) from \cite{FL2} and
letting the indices $k$, $l$, $m$, $n$ run over
$$
m,n = -l, -l+1, -l+2, \dots, \qquad k = -1, -2, -3, \dots, \qquad
l = -1,-\frac32, -2, -\frac52,\dots,
$$
we obtain:
\begin{multline*}
\sum_{k,l,m,n} (2l+1) \tau^l_{n\underline{m}}(W) \cdot N(W)^k \cdot
\tau^l_{m\underline{n}}(Z^{-1}) \cdot N(Z)^{-k-2}  \\
= \sum_{k,l,n} \frac{2l+1}{N(Z)^2} \cdot \tau^l_{n\underline{n}}(WZ^{-1}) \cdot
N(WZ^{-1})^k
= \sum_{k,l} \frac{2l+1}{N(Z)^2} \cdot
\tilde\Theta_l^-(WZ^{-1}) \cdot N(WZ^{-1})^k  \\
= \sum_{k,l} \frac{2l+1}{N(Z)^2} \cdot
\frac{\lambda_1^{2l+1} \cdot (\lambda_1 \lambda_2)^k}{\lambda_1-\lambda_2}
= \sum_{l \le -1} \frac{(2l+1) N(Z)^{-2} \cdot \lambda_1^{2l+1}}
{(\lambda_1-\lambda_2)(\lambda_1 \lambda_2 -1)}  \\
= \frac{N(Z)^{-2} \cdot \lambda_1}
{(\lambda_1-\lambda_2)(1-\lambda_1 \lambda_2)(1-\lambda_1)^2}.
\end{multline*}
\end{proof}

From these expansions of reproducing kernels and orthogonality relations
(\ref{orthogonality}) we immediately obtain the
following extension of Theorem 57 from \cite{FL2}, which can be regarded as
a split analogue of Corollary 14 and Theorem 15 in \cite{FL4}.

\begin{thm}  \label{projector-thm}
Let $R>0$. The integral operators
$$
f \quad \mapsto \quad
\frac{i}{2\pi^3} \int_{Z \in R \cdot U(1,1)} k(Z,W) \cdot f(Z) \,dV,
\qquad f \in {\cal D}^h \oplus {\cal D}^a,
$$
provide the $\mathfrak{gl}(2,\HC)$-equivariant projectors onto the
irreducible components of ${\cal D}^h \oplus {\cal D}^a$ as follows:
\begin{enumerate}
\item
If $k(Z,W)= N(Z-W)^{-2}$ and $W \in R \cdot \Gamma^-$, this operator provides
the projector onto ${\cal D}^{--}$;
\item
If $k(Z,W)= N(Z-W)^{-2}$ and $W \in R \cdot \Gamma^+$, this operator provides
the projector onto ${\cal D}^{++}$;
\item
If $k(Z,W)= \frac{- N(Z)^{-2} \cdot \lambda_1}
{(\lambda_1-\lambda_2)(1-\lambda_1 \lambda_2)(1-\lambda_1)^2}$
and $W \in R \cdot \Gamma^-$, $|N(W)|>1$,
this operator provides the projector onto ${\cal D}^h_<$;
\item
If $k(Z,W)= \frac{-N(Z)^{-2} \cdot \lambda_2}
{(\lambda_1-\lambda_2)(1-\lambda_1 \lambda_2)(1-\lambda_2)^2}$
and $W \in R \cdot \Gamma^-$, $|N(W)|<1$,
this operator provides the projector onto ${\cal D}^h_>$;
\item
If $k(Z,W)= \frac{N(Z)^{-2} \cdot \lambda_2}
{(\lambda_1-\lambda_2)(1-\lambda_1 \lambda_2)(1-\lambda_2)^2}$
and $W \in R \cdot \Gamma^+$, $|N(W)|>1$,
this operator provides the projector onto ${\cal D}^a_<$;
\item
If $k(Z,W)= \frac{N(Z)^{-2} \cdot \lambda_1}
{(\lambda_1-\lambda_2)(1-\lambda_1 \lambda_2)(1-\lambda_1)^2}$
and $W \in R \cdot \Gamma^+$, $|N(W)|<1$,
this operator provides the projector onto ${\cal D}^a_>$;
\end{enumerate}
where $WZ^{-1}$ is $SU(1,1)$-conjugate to $\bigl(\begin{smallmatrix}
\lambda_1 & 0 \\ 0 & \lambda_2 \end{smallmatrix}\bigr)$.
\end{thm}

\separate

\separate

\noindent
{\em Department of Mathematics, Indiana University,
Rawles Hall, 831 East 3rd St, Bloomington, IN 47405}

\end{document}